\documentclass{article}

\usepackage{amsmath,amssymb,amsthm,graphicx,amsfonts} 
\usepackage[margin=1in]{geometry}
\usepackage{enumerate}
\usepackage{tikz-cd} 

\begin{document}

\newtheorem{thm}{Theorem}[section]
\newtheorem{lem}[thm]{Lemma}
\newtheorem{cor}[thm]{Corollary}
\newtheorem*{ques}{Question}
\newtheorem{hyp}{Hypotheses}

\newtheoremstyle{boldremark}
    {\dimexpr\topsep/2\relax} 
    {\dimexpr\topsep/2\relax} 
    {}          
    {}          
    {\bfseries} 
    {.}         
    {.5em}      
    {}          

\theoremstyle{boldremark}
\newtheorem*{rmk}{Remark}
\newtheorem*{nota}{Notation}

\nocite{*} 

\title{Moishezon Spaces and Projectivity Criteria}

\author{David Villalobos-Paz}

\maketitle

\begin{abstract}
We show that a smooth Moishezon space $Y$ is non-projective if and only if it contains a rational curve such that $-[C] \in \overline{\mathrm{NE}}(Y)$. More generally, this holds if $Y$ has $\mathbb{Q}$-factorial, log terminal singularities. We derive this as a consequence of our main technical result: that we can run the relative minimal model program when the base is a normal algebraic space $Y$ of finite type over a field of characteristic $0$. As a second application, we show that every log canonical pair $(Y, \Delta)$, where $Y$ is an algebraic space of finite type over a field of characteristic $0$ admits a dlt modification that is projective over $Y$.
\end{abstract}

\section{Introduction}

A proper, irreducible, reduced analytic space $Y$ is called a Moishezon space if it is bimeromorphic to some projective variety. Any two such projective varieties are birational, and thus $Y$ is endowed with a unique algebraic structure. In fact, Artin showed in \cite{Art} that the category of Moishezon spaces is equivalent to the category of proper algebraic spaces over $\mathbb{C}$. Starting in dimension $3$, there exist Moishezon spaces that are not projective varieties, or even schemes. Hironaka constructed two such examples, which can be found in the Appendix to \cite{Har}. We note that the threefolds in Hironaka's examples contain rational curves whose sum is numerically trivial. 

Peternell conjectured that any Moishezon manifold without rational curves is projective, and in \cite{Pet} he proved this in dimension $3$. In \cite{Sho}, Shokurov proved Peternell's conjecture in arbitrary dimension when $Y$ has mild singularities, assuming the existence and termination of flips. In view of \cite{BCHM}, the assumptions in Shokurov's proof are satisfied. The precise statement is that if $\psi: (Y, \Delta) \to U$ is a proper, but not projective, morphism of normal algebraic spaces, where $(Y, \Delta)$ is analytically $\mathbb{Q}$-factorial and dlt, then $\psi$ contracts a rational curve in $Y$. However, this is clearly not enough to characterise the non-projectivity of $\psi$. In Theorem 5.1.4 of \cite{Kol2}, Koll\'{a}r shows that a proper, algebraic threefold $Y$ with normal, $\mathbb{Q}$-factorial singularities is non-projective if and only if $Y$ contains a curve $C$ such that $-[C] \in \overline{\mathrm{NE}}(Y)$.

We extend Koll\'{a}r's characterisation to arbitrary dimension and to the relative setting, but with an additional restriction on the singularities of $Y$. Our main result is:

\begin{thm} Suppose that $\psi: (Y, \Delta) \to U$ is a proper morphism of normal algebraic spaces of finite type over a field $k$ of characteristic $0$ and that $(Y, \Delta)$ has klt singularities. If $\psi$ is non-projective, then 
\begin{enumerate}
\item either $Y$ contains a rational curve $C$ such that $\psi(C)$ is a point and $-[C] \in \overline{\mathrm{NE}}(Y/U)$, 
\item or $Y$ has a small, $\mathbb{Q}$-factorial modification $Y^{\mathrm{qf}}$ that is projective over $U$ (more precisely, the composite morphism $Y^{\mathrm{qf}} \to Y \to U$ is projective).
\end{enumerate}
\end{thm}

\begin{cor} Suppose that $\psi: Y \to U$ is a proper morphism of normal algebraic spaces of finite type over a field $k$ of characteristic $0$, and that $\Delta$ is some divisor on $Y$ such that the pair $(Y, \Delta)$ has klt singularities. Assume additionally that $Y$ is $\mathbb{Q}$-factorial. Then $\psi$ is non-projective if and only if $Y$ contains a rational curve $C$ such that $\psi(C)$ is a point and $-[C] \in \overline{\mathrm{NE}}(Y/U)$.
\end{cor}

In particular, when $k=\mathbb{C}$, $U = \mathrm{Spec}(k)$, we get:

\begin{cor} Suppose that $Y$ is a Moishezon space with $\mathbb{Q}$-factorial, log terminal singularities. Then $Y$ is non-projective if and only if it contains a rational curve $C$ such that $-[C] \in \overline{\mathrm{NE}}(Y)$.
\end{cor}

As a consequence of Corollary 1.2, we can prove a version of Kleiman's criterion for Moishezon spaces:

\begin{cor} Suppose that $Y$ is a Moishezon space with $\mathbb{Q}$-factorial, log terminal singularities and that $L$ is a Cartier divisor on $Y$. Then $L$ is ample if and only if $L$ has positive degree on every irreducible curve on $Y$ and $L$ induces a strictly positive function on $\overline{\mathrm{NE}}(Y)$.
\end{cor}

In the works of Shokurov and Birkar-Cascini-Hacon-McKernan, they find a log resolution $(X, \Theta)$ of $(Y, \Delta)$ that is projective over $U$ and run a minimal model program on $(X, \Theta)$ over a single \'{e}tale-open affine patch of $U$. The difference in our approach is that we run an MMP for $X$ over $Y$, and this gives us a stronger result. Indeed, our main theorem will follow once we are able to run relative MMPs for $X \to Y$, where the base $Y$ is a normal algebraic space of finite type over a field of characteristic $0$. 

Roughly speaking, we do this as follows: $Y$ is an algebraic space, so it admits \'{e}tale morphisms $Y_j \to Y$ whose images cover $Y$, where the $Y_j$ are affine. We can then run a relative MMP for $X \times_Y Y_j$ over $Y_j$ for each $j$, so it is enough to verify that the resulting minimal models patch together and descend to give a minimal model for $X$ over $Y$. The right way to run the relative MMP for this application is to do as Koll\'{a}r describes in $\cite{Kol}$. 

The technical details are contained in Section 2. The precise MMP statement that we need for our applications is the following:

\begin{thm} Suppose $g: (X, \Theta) \to (Y, \Delta)$ is a projective morphism of algebraic spaces of finite type over a field of characteristic $0$, where $(X, \Theta)$ is $\mathbb{Q}$-factorial and dlt. Suppose that $g$ has exceptional divisor $E= E_1+ \cdots + E_n$ and that $K_X+\Theta \sim_{g, \mathbb{R}} E_\Theta = \Sigma e_jE_j$ for some numbers $e_j \geq 0$. Finally, suppose that $H$ is a divisor on $X$, such that $K_X + \Theta +cH$ is $g$-ample for some number $c$.
Then we may run the $g$-relative $(K_X+\Theta)$-MMP with scaling of $H$. 
\end{thm}

In our applications, we will start with an algebraic space $Y$ of finite type over a field of characteristic $0$. We will have some boundary divisor $\Delta$ on $Y$, such that the pair $(Y, \Delta)$ has mild singularities, and $g: (X, \Theta) \to (Y, \Delta)$ will be a suitably chosen log resolution. We will need a divisor $H$ on $X$, such that $K_X + \Theta + cH$ is ample over $Y$ for some number $c$. If $X$ is taken to be globally projective, we could, for example, choose $H$ to be some ample divisor on $X$. In fact, to prove Theorem 1.1, we take $H$ to be ample and general enough that each step in the resulting MMP arises from the contraction of an extremal ray in the relative cone of curves. See Section 3 for the proof of our projectivity criterion and its corollaries.

In addition to our characterisation of projectivity, Theorem 1.5 also allows us to prove:

\begin{thm} Suppose $(Y, \Delta)$ is a log canonical pair, where $Y$ is an algebraic space of finite type over a field of characteristic $0$. Then $(Y, \Delta)$ admits a dlt modification.
\end{thm}

The proof of Theorem 1.6 is in Section 4. For this, we will want $H$ to have support equal to $\mathrm{Ex}(g)$, which puts us in the setting of \cite{Kol}; in this case, we choose our log resolution so that it carries a $g$-ample, $g$-exceptional divisor. This can be arranged by Theorem 1 in \cite{KW}.

In \cite{Many}, it is shown that, with certain assumptions, the minimal model program runs for threefolds in mixed characteristic, so our results extend to algebraic spaces of dimension $3$ in mixed characteristic. There is also work in progress by Lyu and Murayama, \cite{LM}, with the aim of showing that the MMP can be run over quasi-excellent schemes of characteristic $0$. In view of this, our results would extend to the case where $Y$ is an algebraic space of finite type over a quasi-excellent scheme of characteristic $0$.

\subsection*{Acknowledgements}
The author would like to thank his advisor, J\'{a}nos Koll\'{a}r, for his constant support, encouragement, and guidance, as well as Lena Ji for helpful conversations. Partial financial support was provided by the NSF under J. Koll\'{a}r's grant number DMS-1901855.

\section{Technical Results on the Relative MMP}

Our aim is to run the relative MMP over agebraic spaces, but we first develop our theory in a more abstract setting, where in place of $K_X + \Theta$ we consider an arbitrary $\mathbb{R}$-Cartier divisor $D$. Throughout this paper we will work under the following hypotheses: \\

$(\dagger)$ $V$ will be assumed to be a quasi-excellent, normal scheme, $Y$ a quasi-excellent, normal, irreducible algebraic space, and $X$ a quasi-excellent, normal algebraic space.  $\pi: V \to Y$ will be a universally open and quasi-finite morphism, and $g: X \to Y$ a projective morphism. $D$ and $H$ will denote $\mathbb{R}$-Cartier divisors on $X$, such that for some real number $c$, $D+cH$ is $g$-ample. \\

By decreasing the scaling constant $c$ in $(\dagger)$, we eventually reach a number $r>0$, such that $D+rH$ is $g$-nef but not $g$-ample, and $D+(r+\epsilon)H$ is $g$-ample for all $\epsilon > 0$; note that if we reach $r=0$, then $D$ is already $g$-nef, so we do not do anything. We define $\phi^r : X \to Z$ to be the morphism that contracts all curves $C$ in $X$, such that $C$ is contracted by $g$ and $(D+rH) \cdot C = 0$. These curves are those in some extremal face of the relative cone of curves $\overline{\mathrm{NE}}(X/Y)$. We will assume that the contraction $\phi^r$ and the corresponding step $f^r$ of the relative $D$-MMP with scaling of $H$ exist:

\[ \begin{tikzcd}
X \arrow[rr, dashed, "f^r"] \arrow[ddr, "g", swap] \arrow[dr, "\phi^r"] & & X^r \arrow[ld, "\psi^r", swap] \arrow[ldd, "g^r"] \\
& Z \arrow[d] & \\
& Y &
\end{tikzcd}
\]

If the birational transforms $D^r = (f^r)_\star D$ and $H^r = (f^r)_\star H$ are $\mathbb{R}$-Cartier, then the divisor $D^r + (r-\epsilon)H^r$ on $X^r$ is $g^r$-ample for sufficiently small $\epsilon>0$. Then we obtain the next step of this relative MMP by decreasing $r-\epsilon$ and proceeding as above.

We are interested in how the steps of this MMP behave with respect to base change, so we consider the commutative diagram 

\[ \begin{tikzcd}
X_V = (X \times_Y V)^\nu \arrow[r, "g_V"] \arrow[d, "\pi_V"] & V \arrow[d, "\pi"] \\
X \arrow[r, "g"] & Y
\end{tikzcd}
\]

where the superscript $\nu$ will always denote the normalisation. Note that, in view of our quasi-excellence hypothesis, the normalisation morphism is finite, and hence projective, and so $\pi_V$ is quasi-finite and $g_V$ is projective. Additionally, since $\pi$ is universally open, every connected component of $X_V$ dominates a component of $X$. We can then set $D_V = \pi^{\star}_V(D)$ and $H_V = \pi^{\star}_V(H)$. Observe that if $D+cH$ is $g$-ample, then $D_V + cH_V$ is $g_V$-ample, and so by decreasing $c$ we can similarly define the steps of the relative $D_V$-MMP with scaling of $H_V$, assuming they exist.

We will index the outputs of these MMPs in a non-standard way. The steps of the MMP are typically indexed discretely, but this can lead to the steps of the MMP for $X \to Y$ and those of the MMP for $X_V \to V$ to be indexed differently: It could happen that as we decrease $c$ in $D+cH$, we hit the first number $r$ for which $D+rH$ stops being $g$-ample, forcing us to do an MMP step over $Y$, and yet $D_V+rH_V$ is still $g_V$-ample, so that for the moment no MMP step is performed over $V$. In this situation, the first step of the MMP over $V$ would come later than the first step of the MMP over $Y$.

For this reason, we will index the outputs of the MMP continuously, rather than discretely, using the scaling constant as the index. More formally:

\begin{nota} If it exists, the $r^{th}$ output of the $g$-relative $D$-MMP with scaling of $H$, denoted by $f^r: X \dashrightarrow X^r$, will mean the composite $X \dashrightarrow X^{r_1} \dashrightarrow \cdots \dashrightarrow X^{r_n}$ for numbers $r_1 > \cdots > r_n \geq r$, where each $r_i$ is such that $D^{r_{i-1}}+r_i H^{r_{i-1}}$ is nef, but not ample, over $Y$, $D^{r_{i-1}}+(r_i + \epsilon)H^{r_{i-1}}$ is ample over $Y$, and $D^{r_n} + (r-\epsilon)H^{r_n}$ is ample over $Y$ (where $\epsilon > 0$ is sufficiently small). Each birational map in the composition then corresponds to a relative MMP step as described above.
\end{nota}

We note that when the relative MMP with scaling is run as in \cite{BCHM}, we are allowed to have $r_i = r_{i+1}$, since each MMP step comes from the contraction of an extremal ray in the relative cone of curves. However, in \cite{Kol}, each MMP step is allowed to contract a higher-dimensional face of this cone, and we must have $r_i > r_{i+1}$.

Finally, we emphasise that when we talk about the $r^{th}$ output, we do not mean the output after performing $r$ MMP steps, but rather the output after the scaling constant has been decreased to $r$. For example, if $r$ is such that $D+rH$ on $X$ is still $g$-ample, then the $r^{th}$ output $f^r: X \dashrightarrow X^r$ is simply the identity map on $X$.

Suppose we are in the setting of $(\dagger)$. We first give an abstract characterisation of the $r^{th}$ output of our relative MMP.

\begin{lem} Suppose the hypotheses in $(\dagger)$ are satisfied. The $r^{th}$ output of the relative D-MMP over $Y$ with scaling of $H$, if it exists, is characterised by the following properties:
\begin{enumerate}
\item $f^r : X \dashrightarrow X^r$ is a birational contraction to a normal variety over $Y$,
\item $f^r_{\star} (D + (r - \epsilon)H)$ is ample over $Y$ for sufficiently small $\epsilon >0$.
\item $f^r$ only contracts divisors contained in the relative exceptional locus $\mathbb{E} (D + rH)$ of $D+rH$; that is, it only contracts divisors $E$ for which the restriction $(D+rH)|_E$ is not big over $Y$.
\end{enumerate}
\end{lem}
\begin{proof} That the $r^{th}$ output satisfies these three properties is clear. Suppose that $h: X \dashrightarrow X'$ is another birational map satisfying the same three properties. Consider the composite $f^r \circ h^{-1} : X' \dashrightarrow X^r$. We claim that this is an isomorphism in codimension $1$. For this, it is enough to show that $f^r$ and $h$ contract the same divisors. Suppose $E$ is some divisor contracted by $f^r$. Then $(D+rH)|_E$ is not big over $Y$. Since $h_{\star} (D + (r - \epsilon)H)$ is ample, and in particular big, over $Y$, then $h$ must also contract $E$: Otherwise, $h_{\star}$E would be a divisor on $X^r$ on which $h_{\star} (D + (r - \epsilon)H)$ is not big. Reversing the roles of $f$ and $h$, we see that $X^r$ and $X'$ have the same divisors. \\

We also know that $h_{\star} (D + (r - \epsilon)H)$ is ample over $Y$ and its birational transform by $f^r \circ h^{-1}$, namely $f^r_{\star} (D + (r - \epsilon)H)$, is also ample over $Y$. We thus conclude that $f^r \circ h^{-1} : X' \dashrightarrow X^r$ is an isomorphism. This follows from the relative version of a result of Matsusaka and Mumford; see \cite{MM}.
\end{proof}

\begin{thm} Suppose the hypotheses in $(\dagger)$ are satisfied. If $X^r$ exists, then $(X_V)^r$ exists and $(X_V)^r \cong (X^r \times_Y V)^{\nu}$.
\end{thm}
\begin{proof} We assume the existence of $f^r : X \dashrightarrow X^r$. Consider the normalised base change $f^r_V : X_V \dashrightarrow (X^r \times_Y V)^{\nu}$. We claim that this satisfies the three conditions of Lemma 2.1. Indeed, we have a commutative diagram 

\[ \begin{tikzcd}
X_V = (X \times_Y V)^{\nu} \arrow[r, dashed, "f^r_V"] \arrow[d, "\pi_V"] & (X^r \times_Y V)^{\nu} \arrow[d, "\pi^r_V"] \\
X \arrow[r, dashed, "f^r"] & X^r
\end{tikzcd}
\]

By construction, $(X^r \times_Y V)^{\nu}$ is normal. Suppose that the birational inverse $(f^r_V)^{-1}$ contracts some divisor $E$. Since $\pi^r_V$ is quasi-finite, then $\pi^r_V(E)$ is a divisor in $X^r$ whose image in $X$ by $(f^r)^{-1}$ has codimension greater than $1$. This is a contradiction, because $f^r$ is a birational contraction. Therefore, $f^r_V$ is a birational contraction as well. This gives us Condition 1.\\

Similarly, we know that $(f^r)_{\star}(D+(r-\epsilon)H)$ is ample over $Y$. Since $\pi^r_V$ is quasi-finite, then the pullback $(\pi^r_V)^{\star}(f^r)_{\star}(D+(r-\epsilon)H)$ is ample over $Y$. Because we are dealing with divisors on normal schemes, it is clear that $(\pi^r_V)^{\star}(f^r)_{\star}(D+(r-\epsilon)H) = (f^r_V)_{\star}(D_V + (r - \epsilon)H_V)$, which shows Condition 2. \\

Finally, suppose that $f^r_V$ contracts some divisor $E$. Since $\pi_V$ is quasi-finite, then $\pi_V(E)$ is a divisor in $X$ that must be contracted by $f^r$. Hence, $(D+rH)|_{\pi_V(E)}$ is not big, which implies that $(D_V+rH_V)|_E$ is not big, which is Condition 3.
\end{proof}

\begin{cor} Suppose the hypotheses in $(\dagger)$ are satisfied. Suppose $\{Y_j\}$ is an open cover for $Y$. Define $X_j = X \times_Y Y_j$, $g_j = g|_{X_j}$, $D_j = D|_{Y_j}$, and $H_j = H|_{Y_j}$. Suppose that, for each $j$, we know the existence of the $r^{th}$ output of the $D_j$-MMP with scaling of $H_j$ over $g_j$. Then the $r^{th}$ output of the $D$-MMP with scaling of $H$ over $g$ exists.
\end{cor}
\begin{proof}
We use Theorem 2.2 with $Y=Y_j$, $X=X_j$, $g = g_j$, $D=D_j$, $H=H_j$, $V = Y_j \cap Y_k$, and $\pi: V \to Y_j$ the inclusion, where we let $j,k$ vary over all indices in the cover $\{Y_j\}$. \\

Theorem 2.2 implies that, for all indices $j, k$, we get an $r^{th}$ output $(X_{Y_j} \cap X_{Y_k})^r$ and Lemma 2.1 says that this output is unique. This means exactly that $(X_{Y_j})^r$ and $(X_{Y_k})^r$ are isomorphic over $Y_j \cap Y_k$. Therefore, we may glue these outputs together to obtain a scheme $X^r$ together with a birational map $f^r : X \dashrightarrow X^r$ over $Y$. The three conditions of Lemma 2.1 are satisfied locally on the base of $f^r$, and so $f^r$ itself satisfies them.
\end{proof}

\begin{lem} Suppose that $\tilde{W}$ is a quasi-excellent, normal scheme $\psi: \tilde{W} \to Y$ is a Galois covering with Galois group $G$. We still assume that $Y$ is a quasi-excellent, normal algebraic space and that $g:X \to Y$ is a projective morphism from a normal algebraic space. With $\mathbb{R}$-Cartier divisors $D$ and $H$ on $X$ as before, suppose that we know that the $r^{th}$ output $(f_{\tilde{W}})^r : X_{\tilde{W}} \dashrightarrow (X_{\tilde{W}})^r$ over $\tilde{W}$ exists and is $G$-equivariant. Then its quotient by $G$ gives the $r^{th}$ output of the MMP over $Y$.
\end{lem}
\begin{proof} We need to show that this quotient satisfies the three properties in Lemma 2.1. We have a commutative diagram:

\[ \begin{tikzcd}
X_{\tilde{W}} \arrow[r, dashed, "(f_{\tilde{W}})^r"] \arrow[d, "/G"] & (X_{\tilde{W}})^r \arrow[d, "/G"] \\
X \arrow[r, dashed, "f^r"] & (X_{\tilde{W}})^r/G
\end{tikzcd}
\]

First, we note that the quotient $(X_{\tilde{W}})^r/G$ exists as a normal algebraic space. Suppose that the birational inverse $(f^r)^{-1}$ contracted some divisor $E$. This pulls back to some divisor on $(X_{\tilde{W}})^r$ whose birational transform in $X_{\tilde{W}}$ must have higher codimension, because taking quotients by $G$ gives finite morphisms. This is a contradiction, so $f^r$ must be a birational contraction. \\

Also, $(f^r)_{\star}(D+(r-\epsilon)H)$ is obtained by taking the quotient of $((f_{\tilde{W}})^r)_{\star}(D_{\tilde{W}}+(r-\epsilon)H_{\tilde{W}})$ by $G$. The latter divisor is, by assumption, ample over $\tilde{W}$, and since taking a quotient by $G$ gives a finite morphism, then our divisor over $(X_{\tilde{W}})^r/G$ is ample over $Y$. \\

Finally, suppose that $f^r$ contracts some divisor $E$ for which $(D+rH)|_E$ is big. Then taking pullbacks by the finite quotient map, we obtain that the divisor $\pi_{\tilde{W}}^{\star}E$ is contracted by $(f_{\tilde{W}})^r$ even though $(D_{\tilde{W}} + rH_{\tilde{W}})|_{\pi_{\tilde{W}}^{\star}E}$ is big. This is a contradiction.\\

We are done by Lemma 2.1.
\end{proof}

\begin{thm} Suppose the hypotheses in $(\dagger)$ are satisfied and that $\pi: V \to Y$ is surjective. If $(X_V)^r$ exists, then $X^r$ exists and $(X_V)^r \cong (X^r \times_Y V)^{\nu}$.
\end{thm}
\begin{proof} Since the morphism $\pi : V \to Y$ is quasi-finite, then we can extend it to a finite morphism; that is, there is some scheme $W$ with an open immersion $i: V \to W$ and a surjective, finite morphism $\tilde{\pi}: W \to Y$, such that $\tilde{\pi} \circ i = \pi$. \\

This $\tilde{\pi}$ gives rise to a field extension $k(W)/k(Y)$. Note that $W$ might be disconnected; in this case, by $k(W)$ we mean the compositum of the function fields of the connected components of $W$. We can take a Galois closure $L/k(Y)$ of this field extension. By taking integral closures in $L$, we get a Galois covering $\tilde{W} \to Y$ that dominates $W$. We will denote this by $\phi: \tilde{W} \to W$, such that $\tilde{\pi} \circ \phi: \tilde{W} \to Y$ is Galois. We also note that the intermediate field extension $L/k(W)$ is Galois, and so $\phi$ is itself a Galois covering. \\

Set $\tilde{V} = \phi^{-1}(V)$ and $\tilde{\phi} = \phi|_{\tilde{V}}$. Then $\tilde{\phi}$ is quasi-finite and universally open, so we can apply Theorem 2.2 to conclude the existence of the $r^{th}$ output $(f_{\tilde{V}})^r : X_{\tilde{V}} \dashrightarrow (X_{\tilde{V}})^r \cong ((X_V)^r \times_V \tilde{V})^{\nu}$ of the MMP over $\tilde{V}$. We note that the universal openness of $\tilde{\phi}$ follows from TAG 0F32 in \cite{Stacks}.\\

Put $G = \mathrm{Gal}(L/k(Y))$, so that $G$ acts on $\tilde{W}$ and $\tilde{W}/G \cong Y$. This action is transitive on the fibres of $\tilde{W} \to Y$. Note that since $\pi$ was assumed to be surjective, then the $G$-translates of $\tilde{V}$ cover $\tilde{W}$. In other words, $\{g\tilde{V}\}_{g \in G}$ is an open cover for $\tilde{W}$. Since $g\tilde{V} \cong \tilde{V}$ is an isomorphism for each $g \in G$, and we know that the $r^{th}$ MMP output exists over $\tilde{V}$, then in fact the $r^{th}$ MMP output exists over each open set of our cover. By Corollary 2.3, we deduce the existence of the $r^{th}$ MMP output over $\tilde{W}$, which we will denote by $(f_{\tilde{W}})^r : X_{\tilde{W}} \dashrightarrow (X_{\tilde{W}})^r$. \\

Moreover, we observe that these schemes $X_{\tilde{W}}$ and $(X_{\tilde{W}})^r$ carry $G$-actions, lifted from the $G$-action on $\tilde{W}$, and $(f_{\tilde{W}})^r$ is $G$-equivariant. We now take quotients by $G$, resulting in a birational map $f^r: X \cong X_{\tilde{W}}/G \dashrightarrow (X_{\tilde{W}})^r/G$. By Lemma 2.4, this is the output of the MMP over $Y$ that we wished to construct, so we can write $X^r =  (X_{\tilde{W}})^r/G$.\\

Finally, if we base change this output to $V$ and normalise, we obtain the scheme $(X^r \times_Y V)^{\nu}$. By Theorem 2.2, we know that this satisfies the three conditions of Lemma 2.1, so by uniqueness, we conclude that $(X_V)^r \cong (X^r \times_Y V)^{\nu}$.
\end{proof}

Now we focus on the case where $D = K_X + \Theta$.

\begin{thm} [cf. Theorem 1.5] Suppose $g: (X, \Theta) \to (Y, \Delta)$ is a projective morphism of algebraic spaces of finite type over a field of characteristic $0$, where $(X, \Theta)$ is $\mathbb{Q}$-factorial and dlt. Suppose that $g$ has exceptional divisor $E= E_1+ \cdots + E_n$ and that $K_X+\Theta \sim_{g, \mathbb{R}} E_\Theta = \Sigma e_jE_j$ for some numbers $e_j \geq 0$. Finally, suppose that $H$ is a divisor on $X$, such that $K_X + \Theta +cH$ is $g$-ample for some number $c$.
Then we may run the $g$-relative $(K_X+\Theta)$-MMP with scaling of $H$. 
\end{thm}
\begin{proof} Since $Y$ is an algebraic space of finite type over $k$, we can find finitely many affine $k$-schemes of finite type $Y_j$ along with \'{e}tale morphisms $\pi_j: Y_j \to Y$, such that $Y$ is covered by the images of the $\pi_j$. Let $X_j = X \times_Y Y_j$, which is already normal because $\pi_j$ is \'{e}tale. Let $\Theta_j$ and $H_j$ be the pullbacks of $\Theta$ and $H$ to $X_j$. Note that $K_X$ pulls back to $K_{X_j}$ because $\pi_j$ is \'{e}tale. Since $X$ is $\mathbb{Q}$-factorial, then $E_\Theta$ is $\mathbb{R}$-Cartier. Now $Y_j$ is affine and the pullback of $E_\Theta$ to $X_j$ is effective, exceptional over $Y_j$, and $\mathbb{R}$-Cartier so we can run the relative $(K_{X_j}+\Theta_j)$-MMP over $Y_j$ with scaling of $H_j$ by \cite{Fuj}. \\

For each $j$, the divisor $K_X + \Theta +cH$ pulls back to a divisor on $X_j$ that is ample over $Y_j$, and so we may decrease $c$ to get the MMP started. Let $r$ be the first value of the scaling constant for which $K_X + \Theta +rH$ is no longer ample over $Y_j$ for at least one index $j$. We know that the output $(X_j)^r$ over $Y_j$ exists for each $j$. Let $V$ be the disjoint union of the $Y_j$, so that $X_V$ is the disjoint union of the $X_j$. Then the output $(X_V)^r$ over $V$ exists: indeed, it is the disjoint union of the $(X_j)^r$. We observe that $V \to Y$ is quasi-finite and universally open by TAG 0F32 in \cite{Stacks}. Now by Theorem 2.5, $(X_V)^r$ descends to give the $r^{th}$ output $X^r$ of the relative $(K_X+\Theta)$-MMP over $Y$ with scaling of $H$. \\

We now proceed by decreasing $r$ further to obtain the subsequent MMP outputs. Since the MMP over each of the $Y_j$ eventually terminates, then the MMP over $Y$ terminates, too.
\end{proof}

\section{A Characterisation of Projectivity}

\begin{lem} Suppose that, in addition to the hypotheses of Theorem 2.6, $H = \sum h_iH_i$ is such that the numerical classes of the $H_i$ span $\mathrm{NS}(X)_{\mathbb{Q}}$ and the coefficients $h_i$ are linearly independent over $\mathbb{Q}(e_1, \ldots, e_n)$. Then each step in the MMP of Theorem 2.6 arises from the contraction of an extremal ray in the relative cone of curves.
\end{lem}
\begin{proof} We prove this for the first step $X \dashrightarrow X^r$, and the proof for the other steps is identical. Let $\phi^r: X \to Z$ be the associated contraction. We mimic the argument in $\cite{Kol}$: If $C, C'$ are any two curves contracted by $\phi^r$, then we have $\sum h_i (H_i \cdot C) = -r^{-1} (E_{\Theta} \cdot C$), and similarly for $C'$.\\

Eliminating $r$ gives $\sum h_i((E_{\Theta} \cdot C')(H_i \cdot C) - (E_{\Theta} \cdot C)(H_i \cdot C')) = 0$. Since the coefficients of $H$ are linearly independent over $\mathbb{Q}(e_1, \ldots, e_n)$, then we have $(E_{\Theta} \cdot C')(H_i \cdot C) = (E_{\Theta} \cdot C)(H_i \cdot C')$ for every $i$. Therefore, $C$ and $C'$ are proportional as functions on $\langle H_i \rangle_{\mathbb{R}}$. Since the numerical classes of the $H_i$ span $\mathrm{NS}(X)$, then $C$ and $C'$ are numerically equivalent. This then shows that $\overline{\mathrm{NE}}(X/Z)$ is $1$-dimensional.
\end{proof}

\begin{thm} [cf. Theorem 1.1] Suppose that $\psi: (Y, \Delta) \to U$ is a proper morphism of normal algebraic spaces of finite type over a field $k$ of characteristic $0$ and that $(Y, \Delta)$ has klt singularities. If $\psi$ is non-projective, then 
\begin{enumerate}
\item either $Y$ contains a rational curve $C$ such that $\psi(C)$ is a point and $-[C] \in \overline{\mathrm{NE}}(Y/U)$, 
\item or $Y$ has a small, $\mathbb{Q}$-factorial modification that is projective over $U$. 
\end{enumerate}
\end{thm}
\begin{proof} We first find a log resolution $g: X \to (Y, \Delta)$, such that $X$ is projective over $U$. Write $E = E_1 + \cdots + E_n$ for the exceptional divisor of $g$. We can write $K_X + F_1 = g^{\star}(K_Y + \Delta) +F_2$, where $F_1, F_2$ are effective and $F_2$ is $g$-exceptional. Since $(Y, \Delta)$ is klt, then the coefficients of $F_1$ are all less than $1$. For $0 < \eta \ll 1$, the coefficients of $F_1 + \eta E$ are still less than $1$. Choose such a value of $\eta$, and let $\Theta = F_1 + \eta E$. Then we have $K_X + \Theta \sim_{g, \mathbb{R}} E_{\Theta}$, where $E_{\Theta} = F_2 + \eta E$ is effective and $\mathrm{Supp}(E_{\Theta}) = \mathrm{Ex}(g)$. \\

With this choice of $\Theta$, the pair $(X, \Theta)$ is klt. Now choose a divisor $A$, sufficiently ample over $U$, such that $\mathrm{Supp}(A) \cup \mathrm{Supp}(\Theta)$ is an snc divisor, and such that $K_X + \Theta + cA$ is ample over $Y$, where $0<c<1$ is sufficiently general. We may pick such a divisor by Bertini's Theorem. Note that this choice of $A$ implies that the pair $(X, \Theta + cA)$ is klt. Next, we perturb $cA$ as follows: for every divisor class in some basis for $\mathrm{NS}(X)$, we pick a divisor representing that class and we add a sufficiently small, sufficiently general multiple of it to $cA$. Call the resulting divisor $H$. With these choices, we can arrange that $K_X + \Theta + H$ is ample over $Y$, $(X, \Theta + H)$ is klt, and the coefficients of $H$ are linearly independent over $\mathbb{Q}(e_1, \ldots, e_n$), where $e_1, \ldots, e_n$ are the coefficients of $E_{\Theta}$. \\

We now run the relative $(K_X + \Theta)$-MMP with scaling of $H$ over $Y$, whose steps exist and terminate by Theorem 2.6. Since $(X, \Theta)$ is klt, then this MMP terminates in a klt pair $(X^{\mathrm{min}}, \Theta^{\mathrm{min}})$ projective over $Y$. Since $Y$ is klt, then by Theorem 3.52 in \cite{KM}, $X^{\mathrm{min}}$ is a small modification of $Y$. Note that our choice of $X$ was smooth, and hence $\mathbb{Q}$-factorial. Since each step of this MMP comes from the contraction of an extremal ray, then $X^{\mathrm{min}}$ is still $\mathbb{Q}$-factorial. If $X^{\mathrm{min}}$ is projective over $U$, then $X^{\mathrm{min}}$ is the claimed small, $\mathbb{Q}$-factorial modification of $Y$ that is projective over $U$. \\

Otherwise, along the course of the MMP described above, there is a first step $X^r \dashrightarrow X^{r'}$, such that $X^r$ is projective over $U$ but $X^{r'}$ is not. We will focus on this MMP step; this is not necessarily the first step where the relative MMPs over $Y$ and over $U$ deviate from each other. Note first that since $K_X + \Theta + H$ is ample over $U$, then in particular it is relatively ample, and therefore the index $r_1$ corresponding to the first step in this MMP satisfies $r_1 < 1$. The $\mathbb{Q}$-factoriality of $X$ then implies that $(X, \Theta + r_1H)$ is klt. Additionally, after we do the first MMP step, the output $(X^{r_1}, \Theta^{r_1} + r_1H^{r_1})$ remains klt. Since the next index, say $r_2$, is smaller than $r_1$ and $X^{r_1}$ remains $\mathbb{Q}$-factorial, then in fact $(X^{r_1}, \Theta^{r_1} + r_2H^{r_1})$ is klt. Inductively, we see that right before we do the MMP step that loses projectivity over $U$, we have a klt pair $(X^r, \Theta^r + r'H^r)$ projective over $U$. Since $X^r$ is $\mathbb{Q}$-factorial, we can actually conclude that $(X^r, \Theta^r + (r'-\epsilon)H^r)$ is klt for $0<\epsilon \ll 1$.\\

By Lemma 3.1, our MMP step of interest, $X^r \dashrightarrow X^{r'}$, arises from the contraction of some extremal ray $R$ in $\overline{\mathrm{NE}}(X^r/Y)$. Let $F$ be the minimal extremal face of the larger cone of curves $\overline{\mathrm{NE}}(X^r/U)$ that contains $R$. We have two possibilities to consider: \\

First, suppose that $F$ is itself a ray. If $F$ contains only curves whose images in $Y$ are points, then our step of the relative MMP over $Y$ is actually a step of the relative MMP over $U$. However, the steps of the MMP over $U$ preserve projectivity over $U$. This then gives us a contradiction, because we assumed that $X^{r'}$ is not projective $U$. It may also happen that $F$ contains some curves whose images in $Y$ are curves; in other words, our step of the MMP over $Y$ doesn't contract every curve in $F$. Then there are curves $C, C'$, such that $[C'] = \lambda [C]$ in $\overline{\mathrm{NE}}(X^r/U)$ for some $\lambda > 0$, and such that $g^r(C)$ is a point and $g^r(C')$ is a curve in $Y$. In fact, by Corollary 1.4 in \cite{HM} applied to the contraction morphism $\mathrm{cont}_F: X^r \to Z$ over $U$, we can take the curves $C, C'$ to be rational. This implies that $[g^r(C')] = 0$ in $\overline{\mathrm{NE}}(Y/U)$, so certainly $- [g^r(C')] \in \overline{\mathrm{NE}}(Y/U)$.\\

The second possibility is that $F$ has dimension greater than $1$. Since $R$ is an extremal ray contracted by a step of some MMP, then it is spanned by the class of some rational curve $C$, and $[C] \in F$. Then we can write $[C] = \sum \lambda_j v_j$, where each $v_j \in R_j$ is some vector contained in an extremal ray $R_j$ of $F$, and $\lambda_j \neq 0$. At this point, we do not know that any $v_j$ is the class of some curve in $X^r$. We have that $K_{X^r} + \Theta^r + (r'+\epsilon)H^r$ is relatively ample, and $(K_{X^r} + \Theta^r + r'H^r) \cdot C = 0$. This means that if we decrease $r'$, we get a negative intersection product: $(K_{X^r} + \Theta^r + (r'-\epsilon)H^r) \cdot C < 0$, so there must exist some $j_0$, such that $(K_{X^r} + \Theta^r + (r'-\epsilon)H^r) \cdot v_{j_0} < 0$ for all sufficiently small $\epsilon > 0$. Since we know that $(X, \Theta^r + (r'-\epsilon)H^r)$ is klt and $R_{j_0}$ is a negative extremal ray for this pair, then by \cite{Kaw} this means that $R_{j_0} = \mathbb{R}_{\geq 0}[C_{j_0}]$ for some rational curve $C_{j_0}$ in $X^r$. Replacing $\lambda_{j_0}$ if necessary, we can assume that $v_{j_0} = [C_{j_0}]$. Letting $\epsilon \to 0$, we see that actually $(K_{X^r} + \Theta^r + r'H^r) \cdot C_{j_0} \leq 0$. Suppose for a contradiction that $g^r(C_{j_0})$ is a point. Then $(K_{X^r} + \Theta^r + r'H^r) \cdot C_{j_0} \geq 0$ because the divisor is relatively nef. Combining our two inequalities, we deduce that $(K_{X^r} + \Theta^r + r'H^r) \cdot C_{j_0} = 0$, so that the ray $R_{j_0} \cap \overline{\mathrm{NE}}(X^r/Y)$ gets contracted by this MMP step. This is impossible because this step contracts only the ray $R$. Therefore, $g^r(C_{j_0})$ is a rational curve in $Y$, and $-[g^r(C_{j_0})] = \lambda_{j_0}^{-1} \sum_{j \neq j_0} \lambda_j [g^r(R_j)] \in \overline{\mathrm{NE}}(Y/U)$.
\end{proof}

\begin{rmk} We observe that it really could happen that, in our proof, $F$ is an extremal ray, yet it contains curve classes that are not in $R$. For example, consider $U = \mathrm{Spec}(k)$ and take any threefold with an ordinary double point. Blowing up this double point gives an exceptional divisor isomorphic to $\mathbb{P}^1 \times \mathbb{P}^1$. A step of the absolute MMP would contract this entire $\mathbb{P}^1 \times \mathbb{P}^1$. However, if we ran an MMP over $\mathbb{P}^1$, we would contract only one of the $\mathbb{P}^1$s in the product.
\end{rmk}

\begin{cor} [cf. Corollary 1.2] Suppose that $\psi: Y \to U$ is a proper morphism of normal algebraic spaces of finite type over a field $k$ of characteristic $0$, and that $\Delta$ is some divisor on $Y$, such that the pair $(Y, \Delta)$ has klt singularities. Assume additionally that $Y$ is $\mathbb{Q}$-factorial. Then $\psi$ is non-projective if and only if $Y$ contains a rational curve $C$ such that $\psi(C)$ is a point and $-[C] \in \overline{\mathrm{NE}}(Y/U)$.
\end{cor}
\begin{proof} If $\psi: Y \to U$ is projective, then $Y$ carries some $\psi$-ample divisor $A$. Then $A$ is positive on every curve on $Y$ that is contracted by $\psi$, so $Y$ cannot contain a curve like the one in our statement. \\

Conversely, if $\psi$ is non-projective, then we can run the argument in the proof of Theorem 3.2. Our relative MMP over $Y$ terminates in a small modification $X^{\mathrm{min}}$ that is projective over $U$. However, we are now assuming that $Y$ is $\mathbb{Q}$-factorial, so $Y$ admits no non-trivial small modifications that are projective over it. This implies that $X^{\mathrm{min}} \cong Y$, so in particular $X^{\mathrm{min}}$ is not projective over $U$. Therefore, there must be some MMP step over $Y$ where projectivity over $U$ is lost. We then argue as in the proof of Theorem 3.2 to find the desired rational curve.
\end{proof}

We then deduce the following characterisation of projectivity. In dimension $3$, without the need for klt singularities, this was done by Koll\'{a}r; see Corollary 5.1.5 in \cite{Kol2}.

\begin{cor} Suppose that $\psi: (Y, \Delta) \to U$ is a proper morphism of normal algebraic spaces of finite type over a field $k$ of characteristic $0$ and that $(Y, \Delta)$ has $\mathbb{Q}$-factorial klt singularities. Then $\psi$ is projective if and only if there is a Cartier divisor $L$ on $Y$ such that $L \cdot C > 0$ for every irreducible curve $C$ in $Y$ such that $\psi(C)$ is a point. ($L$ need not be ample over $U$).
\end{cor}
\begin{proof} If $\psi$ is projective, then $Y$ contains a Cartier divisor $L$ ample over $U$, so we are done. Conversely, suppose that there is a Cartier divisor $L$ on $Y$ that is positive on every irreducible curve contracted by $\psi$. Assume for a contradiction that $\psi$ is not projective. By Corollary 3.3, $Y$ contains a rational curve $C$ such that $-[C] \in \overline{\mathrm{NE}}(Y)$, but then for this curve we have $L \cdot (-C) \geq 0$, which is impossible.
\end{proof}

We recall Kleiman's numerical criterion for ampleness. The usual formulation is that if $Y$ is projective over $U$, then a Cartier divisor $L$ on $Y$ is ample over $U$ if and only if $L$ induces a strictly positive function on $\overline{\mathrm{NE}}(Y/U) \setminus \{0\}$. In fact, Kleiman proved a more general version, where $Y$ is only assumed to be quasi-divisorial over $U$; see Theorem IV.4.2 in \cite{Kle}. There are, however, proper algebraic spaces, even smooth ones, that are not quasi-divisorial: see, for example, Hironaka's second example in the Appendix of \cite{Har}. As a consequence of Corollary 3.4, we prove a version of Kleiman's criterion for which we do not need quasi-divisoriality, but instead we require $Y$ to have mild singularities.  Note that this version of Kleiman's criterion also follows from Lemma 21 in \cite{Kol4}.

\begin{cor} [Kleiman's Criterion, cf. Corollary 1.3] Suppose that $\psi: (Y, \Delta) \to U$ is a proper morphism of normal algebraic spaces of finite type over a field $k$ of characteristic $0$ and that $(Y, \Delta)$ has $\mathbb{Q}$-factorial klt singularities. Let $L$ be a Cartier divisor on $Y$. Then $L$ is $\psi$-ample if and only if $L$ has positive degree on every curve $C$ such that $\psi(C)$ is a point and $L$ induces a strictly positive function on $\overline{\mathrm{NE}}(Y/U) \setminus \{0\}$.
\end{cor}
\begin{proof} Suppose first that $L$ is $\psi$-ample. Then $L$ has positive degree on every curve and $\psi$ is projective, so the usual form of Kleiman's Criterion applies. Conversely, if $L$ is positive on every curve, then by Corollary 3.4, $Y$ is projective, so the usual form of Kleiman's Criterion applies, and we deduce that $L$ is ample.
\end{proof}

\begin{rmk} For our version of Kleiman's Criterion, it is not enough to assume that $L$ induces a strictly positive function on $\overline{\mathrm{NE}}(Y/U) \setminus \{0\}$: We really need the assumption that $L$ has positive degree on every curve contracted by $\psi$, or equivalently that $Y$ has no numerically trivial curves that are contracted by $\psi$. See Exercise VI.2.19.3 in \cite{Kol3} for an example where the latter hypothesis is not satisfied and Kleiman's Criterion does not hold.
\end{rmk}

\section{Existence of dlt Modifications}

Our next applications concern the existence of dlt modifications of lc pairs.

\begin{thm} [cf. Theorem 1.6] Suppose $(Y, \Delta)$ is a log canonical pair, where $Y$ is an algebraic space of finite type over a field of characteristic $0$. Then $(Y, \Delta)$ admits a dlt modification that carries a relatively ample, exceptional divisor over $Y$.
\end{thm} 
\begin{proof} Let $g: (X, \Theta) \to (Y, \Delta)$ be a log resolution such that the morphism $g$ is projective and such that $\mathrm{Ex}(g)$ supports a relatively ample divisor. This can be arranged by Theorem 1 in \cite{KW}. Here, $\Theta$ is defined by $K_X + \Theta = K_X + E + \widetilde{\Delta} = g^\star(K_Y+\Delta) + E_\Theta$, where $E = E_1 + \cdots + E_n$  is the exceptional divisor of $g$, $\widetilde{\Delta}$ is the strict transform of $\Delta$, and $E_\Theta = \Sigma e_iE_i$ is some combination of the exceptional divisors. In particular, $e_i  \geq 0$ for each $i$ because $(Y, \Delta)$ is log canonical. Additionally, since $(X, \Theta)$ is log smooth, then it is dlt and each $E_i$ is $\mathbb{Q}$-Cartier. \\

Since $H$ is relatively ample, then there exists a number $c$, such that $K_X + \Theta + cH$ is $g$-ample. We can then wiggle the coefficients of $H$ so that they become linearly independent over $\mathbb{Q}(e_1, \ldots, e_n)$ and relative ampleness is preserved. \\

Note that if $Y$ were assumed to be a scheme, then we would be in the setting of Theorem 1 in \cite{Kol}, and our result would follow from Theorem 2.6. When $Y$ is an algebraic space, we cannot directly apply Koll\'{a}r's result. Additionally, the morphisms $\pi: V \to Y$ and $\pi_V:  X_V \to X$ are no longer open immersions and are instead only known to be \'{e}tale (here, $V$ is as in the proof of Theorem 2.6). Thus, the pullback of an irreducible divisor can have multiple components, and hence the coefficients of $H_V = \pi_V^\star H$ are no longer linearly independent over $\mathbb{Q}(e_1, \ldots, e_n)$, so we also cannot directly apply Theorem 1 in \cite{Kol} to the relative $(K_{X_V} + \Theta_V)$-MMP with scaling of $H_V$. \\

Suppose that the first step of the MMP over $V$ contracts some irreducible divisor $F_0$ contained in some $\pi_V^\star E_0$. We will show that actually all of $\pi_V^\star E_0$ is contracted by this same step. \\

Indeed, suppose $C$ is some curve in $\pi_V^\star E_0$, so that $\pi_V(C)$ is contained in $E_0$. Let $C'$ be the component of $\pi_V^\star(\pi_V(C))$ that is contained in $F_0$. Since $F_0$ is contracted, then $(K_{X_V} + \Theta_V + rH_V) \cdot C' = 0$. By the Projection Formula, we have that $(K_X + \Theta + rH) \cdot \pi_V(C) = 0$. But the morphism $\pi_V$ restricts to a finite, \'{e}tale morphism over $\pi_V^\star(\pi_V(C))$, say of degree $d$. Then we have $(K_{X_V} + \Theta_V + rH_V) \cdot \pi_V^\star(\pi_V(C)) = d (K_X + \Theta + rH) \cdot \pi_V(C) = 0$. Since $C$ is a component of $\pi_V^\star(\pi_V(C))$ and $K_{X_V} + \Theta_V + rH_V$ is relatively nef, then $(K_{X_V} + \Theta_V + rH_V) \cdot C = 0$, so $C$ is contracted. \\

In summary, we have shown that if, in the first MMP step over $V$, we contract an irreducible component of $H$ appearing with coefficient $h_i$, then every component with that same coefficient is contracted as well. The same argument shows that is true for all subsequent steps. This is all we need to be in the setting of Theorem 1 in \cite{Kol}. \\

By Theorem 2.6, we may run the relative $(K_X + \Theta)$-MMP with scaling of $H$ over $Y$. This MMP will now terminate in a relative minimal model $g^{\mathrm{min}}: (X^{\mathrm{min}}, \Theta^{\mathrm{min}}) \to (Y, \Delta)$, and this pair $(X^{\mathrm{min}}, \Theta^{\mathrm{min}})$ is the dlt modification we were seeking. Since the birational transform of each one of the exceptional divisors $E_i$ to $X^{\mathrm{min}}$ is $\mathbb{Q}$-Cartier, then the divisor $K_{X^{\mathrm{min}}} + \Theta^{\mathrm{min}} + cH^{\mathrm{min}}$ on $X^{\mathrm{min}}$ is ample over $Y$ for $c$ sufficiently small.
\end{proof}

\begin{rmk} The dlt modification that we obtain in Theorem 4.1 has the property that it carries a relatively ample exceptional divisor. We could, instead, ask for dlt modifications for which the underlying space is $\mathbb{Q}$-factorial. We observe again that the steps of Koll\'{a}r's MMP need not preserve $\mathbb{Q}$-factoriality because we are allowed to contract extremal faces of dimension $\geq 2$. This means that we need to do some more work to obtain $\mathbb{Q}$-factorial dlt modifications.
\end{rmk}

\begin{cor} Let $(Y, \Delta)$ be an lc pair, where $Y$ is an algebraic space of finite type over a field $k$ of characteristic $0$. Then $(Y, \Delta)$ admits a $\mathbb{Q}$-factorial dlt modification that is projective over $Y$.
\end{cor}
\begin{proof} We first take a log resolution $g: (X, \Theta) \to (Y, \Delta)$ and then run the relative MMP with scaling described in Theorem 4.1 to obtain a dlt modification $g^{\mathrm{min}}: (X^{\mathrm{min}}, \Theta^{\mathrm{min}}) \to (Y, \Delta)$. If $\mathrm{dim}(Y) \leq 2$, then $X^{\mathrm{min}}$ is already $\mathbb{Q}$-factorial, so we only need to look at the case where $\mathrm{dim}(Y) \geq 3$.\\

We claim that the abelian group $\mathrm{Cl}(X^{\mathrm{min}})/\mathrm{Pic}(X^{\mathrm{min}})$ is finitely generated. We show this by induction on the MMP steps. Observe that since $X$ is smooth, then $\mathrm{Cl}(X)/\mathrm{Pic}(X)$ is trivial and so certainly finitely generated. Suppose next that after $j$ MMP steps $\mathrm{Cl}(X^j)/\mathrm{Pic}(X^j)$ is finitely generated, and consider the next step:

\[ \begin{tikzcd}
X^j \arrow[rr, dashed, "f"] \arrow[dr, "\phi"] & & X^{j+1} \arrow[ld, "\psi", swap] \\
& Z
\end{tikzcd}
\]

Here, $\phi$ is the contraction morphism of an extremal face $V$ of $\overline{\mathrm{NE}}(X^j/Y)$, and $\psi$ is a small modification. By the Contraction Theorem, we have a short exact sequence $0 \to \mathrm{Pic}(Z) \to \mathrm{Pic}(X^j) \to \mathbb{Z}^{\mathrm{dim}(V)}$. We also know that $\psi$ induces an inclusion $\mathrm{Pic}(Z) \to \mathrm{Pic}(X^{j+1})$. This means that going from $\mathrm{Pic}(X^j)$ to $\mathrm{Pic}(Z)$ decreases the Picard rank by $\mathrm{dim}(V)$, and then going to $\mathrm{Pic}(X^{j+1})$ may increase the rank. Similarly, going from $\mathrm{Cl}(X^j)$ to $\mathrm{Cl}(X^{j+1})$ decreases the rank by $1$ if $\phi$ contracts a divisor, and there is no change in the class group if $\phi$ does not contract any divisors. So at worst, going from $\mathrm{Cl}(X^j)/\mathrm{Pic}(X^j)$ to $\mathrm{Cl}(X^{j+1})/\mathrm{Pic}(X^{j+1})$ increases the rank by $\mathrm{dim}(V)$, and hence our quotient remains finitely generated. \\

By induction it follows that $\mathrm{Cl}(X^{\mathrm{min}})/\mathrm{Pic}(X^{\mathrm{min}}) \cong \mathbb{Z}^m \oplus A$, where $m$ is a non-negative integer and $A$ is a finite abelian group. Each $\mathbb{Z}$ summand is generated by some divisor class $[D_i], 1 \leq i \leq m$, and these divisors $D_i$ are not $\mathbb{Q}$-Cartier. Blowing up along one of these divisors $D_i$ gives a small morphism because $\mathrm{dim}(X^{\mathrm{min}}) \geq 3$, and it decreases the rank of our abelian group by $1$ because the strict transform of $D_i$ is $\mathbb{Q}$-Cartier. To be more precise, we wish to take the symbolic blow-up by $-D_i$ (see e.g. Remark 6.33 in \cite{KM}).\\

After $m$ blow ups, we have a small modification $h: X^{\mathrm{qf}} \to X^{\mathrm{min}}$, where $\mathrm{Cl}(X^{\mathrm{qf}})/\mathrm{Pic}(X^{\mathrm{qf}})$ is finite, and therefore $X^{\mathrm{qf}}$ is $\mathbb{Q}$-factorial. Now $g^{\mathrm{min}} \circ h : (X^{\mathrm{qf}}, h_\star^{-1} \Theta^{\mathrm{min}}) \to (Y, \Delta)$ is the desired $\mathbb{Q}$-factorial dlt modification.
\end{proof}

Finally, we look at an application of Theorem 2.2 where the morphism $\pi: V \to Y$ is quasi-\'{e}tale.

\begin{cor} Suppose $\pi: V \to Y$ is a universally open, quasi-finite and quasi-\'{e}tale morphism of normal quasi-projective varieties over a field of characteristic $0$. Let $(Y, \Delta)$ be lc, and take a projective log resolution $g: (X, \Theta) \to (Y, \Delta)$, such that there exists a relatively ample divisor $H$ with $\mathrm{Supp}(H) = \mathrm{Ex}(g)$. Let $E = E_1 + \cdots E_n$ be the exceptional divisor of $g$. Then running the relative $(K_X+\Theta)$-MMP with scaling of $H$ and base changing it to $V$ is equivalent to running the relative $(K_{X_V}+\Theta_V)$-MMP with scaling of $H_V$.
\end{cor}
\begin{proof} As before, we take $\Theta$ to be such that $K_X + \Theta \sim_{g, \mathbb{R}} E_\Theta$, where $E_\Theta = \sum e_iE_i$ is $g$-exceptional and effective, so that \cite{Fuj} applies. $H$ is relatively ample, so there exists a number $c$ for which $K_X + \Theta +cH$ is $g$-ample. We can wiggle the coefficiens of $H$ to guarantee that they are linearly independent over $\mathbb{Q}(e_1, \ldots e_n)$. Since $\pi$ is quasi-\'{e}tale, then all the ramification of $\pi_V : X_V \to X$ is along the exceptional divisor $E$. We have a log resolution $g_V: (X_V, \Theta_V) \to (V, \pi^{\star}\Delta)$, where $\Theta_V$ is defined similarly to $\Theta$. By the Riemann-Hurwitz formula, it follows that $\pi_V^\star(K_X+\Theta) = K_{X_V} + \Theta_V$.\\

Now we decrease $c$ until we reach the first value $r$ for which $K_X + \Theta + rH$ is no longer $g$-ample. Then we have $(X_V)^r \cong X^r \times_Y V$ by Theorem 2.2. We can keep on decreasing $r$ until the MMP over $Y$ terminates.
\end{proof}

\begin{rmk} We observe that $\pi_V^\star K_X \neq K_{X_V}$ and $\pi_V^\star \Theta \neq \Theta_V$, and that we need to be in characteristic $0$ to avoid wild ramification. 
\end{rmk}

\begin{rmk} A particular example of Corollary 4.3 is when $\pi: V \to Y$ is a quotient singularity by some group action. In this case, we may run the relative MMP on $g: (X, \Theta) \to (Y, \Delta)$ to obtain a dlt modification of $(Y, \Delta)$. However, the corresponding relative MMP on $g_V: X_V \to V$ will not yield a dlt modification of $(V, \pi^{\star}\Delta)$; this will only be a quotient-dlt modification. See Section 5 of \cite{dFKX} for more on quotient-dlt pairs.
\end{rmk}


\bibliographystyle{plain}
\bibliography{template}

\begin{thebibliography}{99}
\bibitem{Art} Michael Artin, \textit{Algebraization of Formal Moduli: II. Existence of Modifications}, Annals of Mathematics, vol. 91, no. 1, 1970, pp. 88–135. 
\bibitem{Many} Bhargav Bhatt, Linquan Ma, Zsolt Patakfalvi, Karl Schwede, Kevin Tucker, Joe Waldron and Jakub Witaszek, \textit{Globally +-Regular Varieties and the Minimal Model Program for Threefolds in Mixed Characteristic}, arXiv e-prints (2020), arXiv:2012.15801
\bibitem{BCHM} Caucher Birkar, Paolo Cascini, Christopher Hacon and James McKernan, \textit{Existence of minimal models for varieties of log general type}, J. Amer. Math. Soc. 23, no. 2, 2010, pp. 405-468.
\bibitem{dFKX} Tomasso de Fernex, J\'{a}nos Koll\'{a}r and Chenyang Xu, \textit{The Dual Complex of Singularities}, Advanced Studies in Pure Mathematics, vol. 74, 2017, pp. 103-129.
\bibitem{Fuj} Osamu Fujino, \textit{Semi-stable Minimal Model Program for Varieties with Trivial Canonical Divisor}, Proc. Japan Acad. Ser. A Math. Sci., vol. 87, no. 3, 2011, pp. 25-30.
\bibitem{HM} Christopher Hacon and James McKernan, \textit{On Shokurov's Rational Connectedness Conjecture}, Duke Math. J., vol. 138, no. 1, 2007, pp. 119-136.
\bibitem{Har} Robin Hartshorne, \textit{Algebraic Geometry}, Graduate Texts in Mathematics, No. 52. Springer-Verlag, New York-Heidelberg, 1977.
\bibitem{Kaw} Yujiro Kawamata, \textit{On the Length of an Extremal Rational Curve}, Invent. Math, vol. 105, 1991, pp. 609-611.
\bibitem{Kle} Steven Kleiman, \textit{Toward a Numerical Theory of Ampleness}, Annals of Mathematics, vol. 84, no. 3, 1966, pp. 293-344.
\bibitem{Kol2} J\'{a}nos Koll\'{a}r, \textit{Flips, Flops, Minimal Models, etc.}, Surveys in Differential Geometry, Cambridge, MA, 1990, pp. 113-199.
\bibitem{Kol3} J\'{a}nos Koll\'{a}r, \textit{Rational Curves on Algebraic Varieties}, Ergebnisse der Mathematik und ihrer Grenzgebiete, 3. Folge., vol. 32, Springer-Verlag, Berlin, 1996.
\bibitem{Kol} J\'{a}nos Koll\'{a}r, \textit{Relative MMP without} $\mathbb{Q}$\textit{-factoriality}, Electronic Research Archive, 2021 (to appear).
\bibitem{Kol4} J\'{a}nos Koll\'{a}r, \textit{Deformations of Varieties of General Type}, arXiv e-prints (2021), arXiv:2101.10986
\bibitem{KM} J\'{a}nos Koll\'{a}r and Shigefumi Mori, \textit{Birational Geometry of Algebraic Varieties}, Cambridge Tracts in Mathematics, vol. 134. Cambridge University Press, 1998.
\bibitem{KW} J\'{a}nos Koll\'{a}r and Jakub Witaszek, \textit{Resolution and Alteration with Ample Exceptional Divisor}, arXiv e-prints (2021), arXiv:2102.03162
\bibitem{LM} Shiji Lyu and Takumi Murayama, \textit{(in progress)}, (2021).
\bibitem{MM} Teruhisa Matsusaka and David Mumford, \textit{Two Fundamental Theorems on Deformations of Polarized Varieties}, American Journal of Mathematics, vol. 86, no. 3, 1964, pp. 668-684.
\bibitem{Pet} Thomas Peternell, \textit{Rational Curves in Moishezon 3-folds}, Lecture Notes in Math. 1194, 1986, pp. 133-144.
\bibitem{Sho} Vyacheslav Shokurov, \textit{Letters of a bi-rationalist. I. A projectivity criterion}, Birational algebraic geometry (Baltimore, MD, 1996), Contemp. Math., vol. 207, Amer. Math. Soc., Providence, RI, 1997, pp. 143–152.
\bibitem{Stacks} The Stacks Project Authors, \textit{Stacks Project}, http://stacks.math.columbia.edu, 2021.
\end{thebibliography}

\end{document}